\documentclass[11pt,reqno]{amsart}

\textfloatsep = 0.4in \addtocontents{toc}{\vspace{0.4in} \hfill
Page\endgraf} \addtocontents{lof}{\vspace{0.2in} \hspace{0.13in} \
Figure\hfill Page\endgraf} \addtocontents{lot}{\vspace{0.2in}
\hspace{0.13in} \ Table\hfill Page\endgraf}

\usepackage{array}
\usepackage{pdfpages}
\usepackage{oplotsymbl}
\usepackage{todonotes}
\usepackage{yfonts}
\usepackage{hyperref}
\usepackage{pgfplots}
\usepgfplotslibrary{groupplots}
\pgfplotsset{compat=1.6}
\pgfplotsset{every axis title/.append style={at={(0.6,1.1)}}}
\usepackage{tcolorbox}
\tcbuselibrary{skins}
\usepackage{algorithm}
\usepackage{algorithmic}
\usepackage{faktor}
\usepackage{array} 
\usepackage{pgf, tikz} 	
\usetikzlibrary{automata, positioning, arrows, decorations.pathreplacing, decorations.pathmorphing}

\definecolor{brightlavender}{rgb}{0.75, 0.58, 0.89}
\definecolor{amethyst}{rgb}{0.6, 0.4, 0.8} 	
\definecolor{blue-violet}{rgb}{0.54, 0.17, 0.89} 	

\definecolor{burgundy}{rgb}{0.5, 0.0, 0.13} 	
\definecolor{burntorange}{rgb}{0.93, 0.53, 0.18} 	
\definecolor{earthyellow}{rgb}{0.88, 0.66, 0.37}
\definecolor{applegreen}{rgb}{0.55, 0.71, 0.0}	
\definecolor{antiquefuchsia}{rgb}{0.57, 0.36, 0.51}
\definecolor{ultramarine}{rgb}{0.07, 0.04, 0.56}

\usepackage{bbm}

\usepackage[utf8]{inputenc}

\usepackage{graphicx}

\usepackage{hyperref}
\hypersetup{
    colorlinks = true,
    citecolor=blue,
    filecolor=black,
    linkcolor=blue,
    urlcolor=black
}

\usepackage{caption}
\usepackage{subcaption}
\usepackage{tikz-network}
 	
\usepackage{amsthm}
\usepackage{amssymb}
\usepackage{xfrac}
\usepackage{latexsym}
\usepackage{amsbsy}
\usepackage{amsfonts}
\usepackage{pstricks}
\usepackage{graphicx}
\usepackage[utf8]{inputenc}
\usepackage{enumitem}

\numberwithin{equation}{section}
\newtheorem{theorem}{Theorem}[section]

\newtheorem{proposition}[theorem]{Proposition}
\newtheorem{lemma}[theorem]{Lemma}

\newtheorem{definition}[theorem]{Definition}

\usepackage[backend=bibtex, doi=false, isbn=false, url=false]{biblatex}

\begin{document}

\title[MCMC sampling of densest $k$-subgraphs]{An MCMC sampling of densest $k$-sub-graphs of regular graphs with connected complement graph}

\author[J. W. Fischer]{\textbf{\quad {Jens Walter} Fischer$^{\spadesuit\clubsuit}$ \, \, }}
\address{{\bf {Jens Walter} FISCHER},\\ Institut de Math\'ematiques de Toulouse. CNRS UMR 5219. \\
Universit\'e Paul Sabatier
\\ 118 route
de Narbonne, F-31062 Toulouse cedex 09.} \email{jens.fischer@math.univ-toulouse.fr}

\begin{abstract}
	We present an exclusion process based approach for sampling densest $k$-sub-graphs from regular graphs $L$ with connected complement. By interpreting an exclusion process as a Markov chain on a corresponding Token Graph $\mathfrak{L}_k$, we make use of classical Markov chain theory to obtain quantitative bounds on the convergence speed depending on the geometry of $L$, which are sharp in the sens that variations in the valency lead to an equality, which is $1$, for the geometric bounds in the case of a complete graph. We propose an algorithm which makes use of this particle view to avoid excessive memory use due to the state space size and discuss the regularity and connectivity condition on $L$ and $L^c$ in the Outlook.
\end{abstract}
\bigskip

\maketitle

\textsc{$^{\spadesuit}$  Universit\'e de Toulouse}
\smallskip

\textsc{$^{\clubsuit}$ University of Potsdam}
\smallskip

\bigskip


\section{Introduction}	
	In this work we use generalized exclusion processes to approach the task of finding $k$-densest sub-graphs in a $\bar{d}$-regular graph $L$ of arbitrary size. This is, in particular, linked to the community detection problem which has received wide attention in modern research, for example in \cite{Newman2002} and \cite{Fortunato2010}. The density of any vertex induced sub-graph $L_{\mathfrak{v}}=(\mathfrak{v},E_{\mathfrak{v}})$ of some graph $L=(V,E)$ on the subset $\mathfrak{v}\subseteq V$ is defined as $\rho(\mathfrak{v})=\frac{|E_{\mathfrak{v}}|}{|\mathfrak{v}|}$, see for example \cite{Charikar2000}. The problem of finding the densest sub-graph becomes, then, the maximization problem which aims at finding the subset $\mathfrak{v}^{\ast}\subset V$ such that $\rho(\mathfrak{v}^{\ast})=\max_{\mathfrak{v}\in \mathcal{P}(V),\; |\mathfrak{v}|\geq 1}\rho(\mathfrak{v})$ where $\mathcal{P}(V)$ is the power set of $V$. In \cite{Charikar2000} the author shows that this problem can be solved in polynomial time. Fixing the size of the sub-graph renders the problem vastly more complicated as being discussed in, among many others, \cite{Corneil1984}, \cite{Feige2001} and \cite{Khuller2009}. In particular, the work in \cite{Khuller2009} and \cite{Feige2001}, the authors conclude, that finding densest sub-graphs of fixed size is NP-hard and finding such graphs is, therefore, unfeasible. Hence, the research focus is on approximations of such objects using combinatorial selection methods or related families of problems which, reliably, give results which are close in density to the desired sub-graph.\par
We contribute to this discussion a stochastic approach based on generalized exclusion processes and Monte Carlo simulation. The process is constructed on the complement graph $L^c$ in such a way, that it exhibits in the long time limit most likely to a densest sub-graph in $L$ and, even, gives monotonously sub-graphs of lesser density, i.e., a second densest sub-graph is obtained with second largest probability, a third densest with third largest probability and so on. We obtain a quantitative bound, which is tight with respect to variations in the average degree $\bar{d}$, on the convergence speed based on a binary geometric condition on $L$.
\section{An appropriate GEP}
	To identify densest sub-graphs we want to construct a particle system based on discrete time generalized exclusion process (GEP) $\eta_k$ using the following Definition \ref{def:generalized_exclusion_process}. 
\begin{definition}\label{def:generalized_exclusion_process}
	Let $L = (V,E)$ be a simple connected graph with $|V|=\bar{n}\in\mathbb{N}$ and $k\in\{1,\hdots,\bar{n}-1\}$. Denote by $(P^{(\eta)})_{\eta\in \{0,1\}^V,\;|\eta|=k}$ a family of stochastic matrices. A GEP in discrete time $\eta_k:=(\eta_{k;t})_{t\in\mathbb{N}}$ of $k$ particles on $L$ is a Markov chain on the set of configurations 
	\begin{equation*}
		\mathcal{S}_k=\{\eta\in\{0,1\}^V|\;|\eta|=k\}
	\end{equation*} 
	defined by the transition matrix $Q=(q_{\eta,\mu})_{\eta,\mu\in \{0,1\}^V}$ given for $\eta,\mu\in\{0,1\}^V$ by
	\begin{equation}
		q_{\eta,\mu} = \begin{cases} P^{(\eta)}(v,w)\mathbbm{1}_{\substack{\eta(v)=1=\mu(w),\eta(w)=0=\mu(v)\\ \eta(u)=\mu(v) \forall u\not\in \{v,w\}}},& \eta\neq\mu \vspace{5pt}\\
		1- \sum_{\mu \neq \eta} q_{\eta,\mu} ,&\eta = \mu.
					   \end{cases} 
	\end{equation}         
\end{definition}
We are going to exploit a link between a system of moving particles $\eta_k$ on a finite graph and a Markov chain in a higher dimensional space, as developed in \cite{Fischer2022}, to have access to classical results on the long time behavior of the process. In particular, we consider discrete time generalized exclusion processes, choosing the transition probabilities according to the application at hand. This gives us access to a process which respects at any time step the fixed size of the subgraphs in which we are interested. We translate GEPs as defined hereinabove to Markov chains on subsets of fixed size of the vertex set of the underlying graph $L$. The arising state space corresponds to what is know as Token Graphs or Symmetric Powers of Graphs, as defined in \cite{Alavi2002} and \cite{FabilaMonroy2012}.
\begin{definition}[\cite{FabilaMonroy2012}]\label{def:mathfrak_L_k}
	Consider a finite simple graph $L=(V,E)$. Define $\mathfrak{V}_k=\{\mathfrak{v}\subset v|\;|\mathfrak{v}|=k\}$ and $\mathfrak{E}_k=\{\langle\mathfrak{v},\mathfrak{w}\rangle|\mathfrak{v}\triangle\mathfrak{w}=\{v,w\};\langle v,w\rangle\in E\}$. The graph $\mathfrak{L}_k=(\mathfrak{V}_k,\mathfrak{E}_k)$  is called the $k$ Token Graph of $L$ and we denote by $\mathrm{deg}_k(\mathfrak{v})$ the degree of $\mathfrak{v}\in\mathfrak{V}_k$ in $\mathfrak{L}_k$.
\end{definition}
In choosing an adequate family of transition matrices $(P^{\theta})_{\theta\in\{0,1\}^{|V|}}$ we obtain a GEP on some graph $L$, based on which we can propose and quantify a Markov chain Monte Carlo approach to sampling densest sub-graphs. This GEP $\eta_k$ has a discrete time Markov chain representation $\mathfrak{S}_k^{MC}$ on $\mathfrak{L}_k$ as developed in \cite{Fischer2022}. Consider as an example the graph illustrated in Figure \ref{fig:bipartite_graph_interpretation_loop}. Indeed, there are various ways to approach GEPs, depending on the distributions which govern the transitions of individual particles. One way is by defining a Markov chain based on a series of dynamic graphs, i.e., by interpreting $\eta_k$ as a Markov chain in a random environment given by the family $(P^{\eta})_{\eta\in \{0,1\}^V,\;|\eta|=k}$. To this end, define  for $t\in\mathbb{N}$ the set $\mathfrak{N}_{k;t}=\{v\in V| \eta_{k;t}(v) = 1\}$ and the directed bipartite graph $B = (B_t)_{t\in\mathbb{N}}$ with $B_t = ((\mathfrak{N}_{k;t}, V\backslash \mathfrak{N}_{k;t}, \Sigma_{k,t})$, where $\Sigma_{k,t}$ represent all possible transitions of particles at time $t$, with potential loops on vertices in $\mathfrak{N}_{k;t}$. The random graph $B_t$ might be disconnected but serves as a graph theoretical representation of transitions of the GEP $\eta_k$, i.e., the positive transition probabilities of $P^{\eta}$ conditioned on $\eta$. \par
The central part is the change of $\Sigma_{k,t}$ when going from time $t$ to $t+1$. It captures the possible particle movements induced by the distribution used in Definition \ref{def:generalized_exclusion_process}. To illustrate the construction, we consider a $3$-regular graph on $6$ vertices depicted in Figure \ref{fig:three_reg_graph_underlying}. It will be the reference for missing edges which will illustrate the distinctive parts of each exclusion process. To any $\eta\in\{0,1\}^V$ we can associate a $\mathfrak{v}_{\eta}\subset V$ with $\mathfrak{v}_{\eta} = \{v\in V|\eta(v) = 1\}$ and in this sense we can write with a little abuse of notation $\eta = \mathfrak{v}_{\eta}$. For our central construction, assume that for some $\mathfrak{v}\subset V$ with $|\mathfrak{v}|=k$ at time $t\in\mathbb{N}$ the process satisfies $\eta_{k;t}=\mathfrak{v}$. Set $\mathfrak{v}^c:=V\backslash\mathfrak{v}$. Consider the possibly disconnected bipartite sub-graph $L_{\mathfrak{v},\mathfrak{v}^c}=((\mathfrak{v},\mathfrak{v}^c),E_{\mathfrak{v},\mathfrak{v}^c})$ of $L$ and add a loop $e_{v,v}$ to any $v\in\mathfrak{v}$. We call the resulting graph $B_t' = (V, \Sigma_{k,t}')$ with $ \Sigma_{k,t}'= E_{\mathfrak{v},\mathfrak{v}^c}\sqcup \{e_{v,v}|v\in\mathfrak{v}\}$. Indeed, for a $\bar{d}$-regular graph $L$ we obtain $|\Sigma_{k,t}|=\mathrm{deg}_k(\mathfrak{v})+k$, where $\mathrm{deg}_k(\mathfrak{v})$ denotes the degree of $\mathfrak{v}$ interpreted as vertex in the token graph $\mathfrak{L}_k$ associated to $L$, see \cite{Alavi2002} and \cite{FabilaMonroy2012}for a deeper discussion of Token Graphs. In Figure \ref{fig:bipartite_graph_interpretation_loop} we illustrate a possible example derived from the graph $L$ as in Figure \ref{fig:three_reg_graph_underlying} and some configuration $\mathfrak{v}$.
\begin{figure}
	\centering
	\begin{subfigure}[!htb]{0.4\textwidth}
		\begin{tikzpicture}[scale = 0.7]
	    	\Vertex[x=-2,y=0]{a1}
			\Vertex[x=-4,y=2]{a2}
			\Vertex[x=-2,y=4]{a3}
			\Vertex[x=2,y=0]{b1}
			\Vertex[x=4,y=2]{b2}
			\Vertex[x=2,y=4]{b3}
			\Edge(a1)(b1)
			\Edge(a1)(b2)
			\Edge(a2)(b1)
			\Edge(a3)(b3)
			\Edge(a3)(b2)
			\Edge(a1)(a2)
			\Edge(a3)(a2)
			\Edge(b3)(b1)
			\Edge(b3)(b2);
		\end{tikzpicture}
		\caption{The underlying $3$-regular graph on $6$ vertices which we are going to use as example for all constructions of exclusion processes.\label{fig:three_reg_graph_underlying}}
	\end{subfigure}\hfill
	\begin{subfigure}[!htb]{0.55\textwidth}
		\begin{tikzpicture}[scale = 0.8]
	    	\Vertex[x=-2,y=0]{a1}
			\Vertex[x=-4,y=2]{a2}
			\Vertex[x=-2,y=4]{a3}
			\Vertex[x=2,y=0]{b1}
			\Vertex[x=4,y=2]{b2}
			\Vertex[x=2,y=4]{b3}
			\Vertex[x=-2,y=0,color = blue,size = 0.4]{z}
			\Vertex[x=-4,y=2,color = blue,size = 0.4]{z}
			\Vertex[x=-2,y=4,color = blue,size = 0.4]{z}
			\Edge[Direct](a1)(b1)
			\Edge[Direct](a1)(b2)
			\Edge[Direct](a2)(b1)
			\Edge[Direct](a3)(b3)
			\Edge[Direct](a3)(b2)
			\Edge[style = dashed](a1)(a2)
			\Edge[style = dashed](a3)(a2)
			\Edge[style = dashed](b3)(b1)
			\Edge[style = dashed](b3)(b2);
			\Edge[loopposition=180,loopshape=45,Direct](a3)(a3)
			\Edge[loopposition=180,loopshape=45,Direct](a1)(a1)
			\Edge[loopposition=180,loopshape=45,Direct](a2)(a2);
		\end{tikzpicture}
		\caption{Potential particle displacements; Uniform distribution for drawing a directed edge. States of the vertices connected by said edge are interchanged. When drawing a loop, everything remains the same. \label{fig:bipartite_graph_interpretation_loop}}
	\end{subfigure}
\end{figure}
In each time step, the exclusion process makes a step by drawing one edge uniformly from $\Sigma_{k,t}'$ and exchanging the state of the endpoints. Note, that in changing the number of present edges also the transition probabilities change due to the uniform distribution over $\Sigma_{k,t}'$. This renders the process, in particular, in-homogeneous because the transition probabilities now depend on the current configuration $\mathfrak{v}$. A configuration $\mathfrak{v}$ can be seen as a representative of the sub-graph class of vertex induced sub-graphs which are spanned by a subset of the original vertex set and all edges included therein.
\begin{definition}
	Let $L=(V,E)$ be any simple graph and $\mathfrak{v}\subseteq V$. Then the graph $L_{\mathfrak{v}}= (\mathfrak{v},E_{\mathfrak{v}})$ with $\langle v,w\rangle\in E_{\mathfrak{v}}$ if and only if $v,w\in\mathfrak{v}$ and $\langle v,w\rangle\in E$ is called the vertex induced sub-graph of $L$ on $\mathfrak{v}$. 
\end{definition}
In what follows, in order to obtain meaningful results we work under the assumption that both $L$ and $L^c$ are connected graphs, as explained in the Introduction. A further discussion of the limitations is postponed to the Outlook. If this assumption is removed, the additional difficulty of initial condition dependencies on disconnected graphs for particle systems has to be considered. We leave this open to further research and remain with the assumption on $L$ and $L^c$, giving some pointers in the Outlook. 
\section{Convergence Results}
	We consider the generalized exclusion process $\eta_k$ as constructed above. Then, by \cite{Fischer2022}, there is an associated Markov chain $\mathfrak{S}_k^{MC}$ on $\mathfrak{L}_k$ such that $\eta_k$ and $\mathfrak{S}_k^{MC}$ are equal in law. We make, in what follows, exclusively claims about the associated Markov chain $\mathfrak{S}_k^{MC}$ on $\mathfrak{L}_k$. 
\begin{theorem}\label{thm:stat_dist_S_BE}
	Let $k\in\{1,\hdots,\bar{n}-1\}$ and consider the Markov chain $\mathcal{S}^{MC}_k$ on $\mathfrak{L}_k=(\mathfrak{V}_k,\mathfrak{E}_k)$. Then it is aperiodic, irreducible and, hence, ergodic. Furthermore, it is a reversible chain and the stationary distribution $\pi_k^{MC}$ is given in terms of $\mathfrak{v}\in\mathfrak{V}_k$ by
	\begin{equation}
		\pi_k^{MC}(\mathfrak{v}) = \dfrac{\mathrm{deg}_k(\mathfrak{v})+k}{2|\mathfrak{E}_k|+k\binom{\bar{n}}{k}}
	\end{equation}			 
\end{theorem}
Based on the transition matrix given by Proposition \ref{prop:exclusion_process_with_loops} the stationary distribution $\pi_k^{MC}$ of $\mathfrak{S}_k^{MC}$ can be derived directly and reversibility follows as well. Indeed, one can identify $\mathfrak{S}_k^{MC}$ with a random walk on $\mathfrak{L}_k$ where any vertex in $\mathfrak{L}_k$ has additionally to its incident edges $k$ loops. Under a dichotomy of geometric conditions on the underlying graph $L$, we find the following convergence speed of $\mathfrak{S}_k^{MC}$ to $\pi_k^{MC}$.
\begin{theorem}\label{thm:convergence_speed_S_MC_evolving_sets}
	Let $L$ be a $\bar{d}$-regular graph on $\bar{n}$ vertices and $k\in\{1,\hdots,\bar{n}-1\}$. Consider the Markov chain $\mathfrak{S}_k^{MC}$ on $\mathfrak{L}_k$ and $\varepsilon>0$. Additionally, define 
	\begin{align*}
		\rho(\bar{n},\bar{d},k) &:= \dfrac{4(\bar{n}-1)^2(\bar{n}-k)^2}{\bar{d}^4}\vspace{5pt}\\
		\xi(\bar{n},\bar{d},k) &:= \log\left(\binom{\bar{n}}{k}\right)+\log\left(\dfrac{k(\bar{n}-k)}{\bar{n}-1} + \dfrac{k}{\bar{d}}\right).
	\end{align*}
	If $L$ is such that $\min_{\mathfrak{v}\in\mathfrak{V}_k}\mathrm{avg\;deg}_k(L_{\mathfrak{v}}) < \bar{d}-1$ then for $\mathfrak{v},\mathfrak{w}\in\mathfrak{V}_k$ and
	\begin{equation}
		t\geq 1+ \rho(\bar{n},\bar{d},k)\left(\log\left(\dfrac{4}{\varepsilon}\right) + \xi(\bar{n},\bar{d},k)\right)
	\end{equation}
	the Markov chain $\mathfrak{S}_k^{MC}$ satisfies
	\begin{equation*}
		\left|\dfrac{\left(p_{k;\mathfrak{v},\mathfrak{w}}^{MC}\right)^{(t)} -\pi_k^{MC}(\mathfrak{w})}{\pi_k^{MC}(\mathfrak{w})}\right|\leq \varepsilon.
	\end{equation*}
	On the other hand, if $L$ is such that $\min_{\mathfrak{v}\in\mathfrak{V}_k}\mathrm{avg\;deg}_k(L_{\mathfrak{v}}) \geq \bar{d}-1$ then there is a constant $C>0$ such that its $\varepsilon$ mixing time w.r.t. the total variation distance is bounded by
	\begin{equation}
		\tau_{mix}(\varepsilon) \leq C\log_2(\varepsilon^{-1})\left(\xi(\bar{n},\bar{d},k)-\dfrac{(\bar{n}-1)^2}{\bar{d}^4}\right) .
	\end{equation}	
\end{theorem}
Note that the constants we used in Theorem \ref{thm:convergence_speed_S_MC_evolving_sets}, namely $\rho(\bar{n},\bar{d},k)$ and  $\xi(\bar{n},\bar{d},k)$, grow at most like a polynomial as $\bar{n}\to\infty$ since $\log\left(\binom{\bar{n}}{k}\right)$ grows at most like $\bar{n}$ by Stirling's approximation and the remaining terms are polynomial in $\bar{n}$. Consequently, $\mathfrak{S}_k^{MC}$ mixes rapidly relative to the size of its actual state space, employing this term as used in \cite{Sinclair1992}. Its mixing is, therefore, polynomial fast in terms of the size of the underlying graph $L$, which makes it a suitable process for sampling sub-graphs of $L$.\par
In the following section, we propose an algorithm and the idea behind sampling densest sub-graphs of $L$ with high probability using the Markov chain $\mathfrak{S}_k^{MC}$ introduced in this section and exploiting its properties. 
\section{Particle based MCMC algorithm for sampling densest $k$-sub-graphs}
	By the structure of the stationary distribution we obtain an ordering based on the degree of the vertices, where the probability of drawing a densest sub-graph as $t\to\infty$ is smallest since $\mathrm{deg}_k(\mathfrak{v})=k\bar{d}-2|E_{\mathfrak{v}}|$. On the other hand, this also implies that it is most probable under $\pi_k^{MC}$ to draw a least dense sub-graph of $L$. Since the densest $k$ sub-graph in $L$ corresponds to the least dense $k$ sub-graph in $L^c$ and vice versa, we can exploit the fact that both are connected to inverse the roles under $\mathfrak{S}_{k}^{MC}$ with respect to the weights given by its stationary distribution. Using this, we propose the following MCMC approach to finding densest $k$ sub-graphs in a $\bar{d}$-regular graph with high probability. To this end consider a $\bar{d}$-regular simple connected graph $L$ on $\bar{n}$ vertices and assume that its graph complement $L^c$ is also connected. Then $L^c$ is a $\bar{n}-1-\bar{d}$ regular simple connected graph. We denote for any $k\in\{1,\hdots,\bar{n}-1\}$ by $\mathfrak{L}_k^{(c)}=(\mathfrak{V}_k, \mathfrak{E}_k^{(c)})$.\par
Let $k\in\{1,\hdots,\bar{n}-1\}$ and define $\mathfrak{S}_{k}^{MC}$ as the Markov chain associated to the exclusion process on $L^c$. Then, by Theorem \ref{thm:stat_dist_S_BE} the Markov chain $\mathfrak{S}_{k}^{MC}$ converges in distribution to $\pi_k^{MC}$ and with maximal probability we obtain a $\mathfrak{v}_{\ast}\in\mathfrak{V}_k$  which induces a least dense $k$-sub-graph $L_{\mathfrak{v}_{\ast}}^c$ of $L^c$ and, in turn, the set $\mathfrak{v}_{\ast}$ induces a densest $k$-sub-graph in $L$. We propose the following algorithm which only uses the local information of $L$ and the current configuration $\mathfrak{v}$ of $\mathfrak{S}_k^{MC}$. For the simulation, it is not necessary to construct the whole state space $\mathfrak{L}_k$ with $\binom{\bar{n}}{k}$ vertices and the size of the edge set given by Proposition \ref{prop:size_edge_set_E_k}. Algorithm \ref{alg:local_algo_S_k_MC} only depends on the neighborhood of all $v\in\mathfrak{v}$ after having fixed a $\mathfrak{v}\in\mathfrak{V}_k$. Exploiting the fact that $L$ is assumed to be $\bar{d}$-regular, we can bound the number of states we have to access in each turn by $\bar{d}\cdot k$ which is also a very crude upper bound for $\mathrm{deg}_k(\mathfrak{v})$. We use multi-sets to describe the algorithm and when we set $\mathfrak{X}=\emptyset$ as a multi-set, we impose that $(\mathfrak{X}\cup\{\mathfrak{v}\})\cup\{\mathfrak{v}\}=\mathfrak{X}\cup\{\mathfrak{v},\mathfrak{v}\}$. The multi-set $\mathfrak{X}$ can be replaced in real code by any mutable list-type object which allows multiple times the same entry. 
\begin{algorithm}
	\caption{Particle based algorithm to simulate $\mathfrak{S}_{k}^{MC}$.}
	\label{alg:local_algo_S_k_MC}
	\begin{algorithmic}
		\STATE\;
		\REQUIRE Terminal time for simulation, e.g. $t_{mix}$, number of trials $m$, graph $L$
		\STATE Initialize $t = 0$\; 
		\STATE Define $\mathfrak{X}=\emptyset$ as multi-set\;
		\STATE Run following burn in step\;
		\STATE Draw initial state $\mathfrak{v}\in \mathfrak{V}_k$\; 
		\STATE Set $\mathfrak{S}_{k,0}^{MC}=\mathfrak{v}$\;
		\WHILE{$t\leq t_{mix}$}
			\STATE Construct bipartite graph $B'(\mathfrak{v})$ as in Figure \ref{fig:bipartite_graph_interpretation_loop} \; 
			\STATE Draw uniformly an edge $\langle v,w\rangle$ from $B'(\mathfrak{v})$\; 
			\STATE Update $\mathfrak{S}_{k,t+1}^{MC} = (\mathfrak{v}\backslash\{v\})\cup\{w\}$; $t = t + 1$\; 
		\ENDWHILE
		\STATE Run sampling after burn in step\;
		\WHILE{$i\leq m$}
			\STATE Construct bipartite graph $B'(\mathfrak{v})$ as in Figure \ref{fig:bipartite_graph_interpretation_loop} \; 
			\STATE Draw uniformly an edge $\langle v,w\rangle$ from $B'(\mathfrak{v})$\; 
			\STATE Update $\mathfrak{S}_{k,t+1}^{MC} = (\mathfrak{v}\backslash\{v\})\cup\{w\}$; $t = t + 1$\; 
			\STATE Update $i = i+1$;
			\STATE Update $\mathfrak{X}=\mathfrak{X}\cup \mathfrak{S}_{k,t}$\;
		\ENDWHILE
		\RETURN $\mathfrak{X}$, statistic of $m$ final states after $t_{mix}$ simulation steps.
	\end{algorithmic}
\end{algorithm}    
Algorithm \ref{alg:local_algo_S_k_MC} mirrors the dynamics of the exclusion process, constructed earlier. It, therefore, samples rapidly, in the sense of \cite{Sinclair1992}, densest sub-graphs with high probability. We want to emphasize the importance of a detailed understanding of the underlying state space $\mathfrak{L}_k$ from the perspective of vertex induced sub-graphs. In particular, the geometric results as well as knowledge about the structure of the vertex set presented in the appendix allowed for the bounds in Theorem \ref{thm:convergence_speed_S_MC_evolving_sets}. Finally, the proposed algorithm benefits from the particle perspective in that it only needs to consider at most $k\bar{d}$ vertices in each step which is a considerable reduction from considering the whole vertex set $\mathfrak{V}_k$. 
\section{Proofs of central results}
	We consider the generalized exclusion process $\eta_k$ as constructed above. Then, by \cite{Fischer2022}, there is an associated Markov chain $\mathfrak{S}_k^{MC}$ on $\mathfrak{L}_k$ such that the marginal distributions of $\eta_k$ and $\mathfrak{S}_k^{MC}$ are identical. To start, we characterize the associated Markov chain $\mathfrak{S}_k^{MC}$ on $\mathfrak{L}_k$ by giving its transition matrix explicitly. 
\begin{proposition}\label{prop:exclusion_process_with_loops}
	Let $k\in\{1,\hdots,\bar{n}-1\}$ and consider the Markov chain $\mathfrak{S}_k^{MC}$ on $\mathfrak{L}_k=(\mathfrak{V}_k,\mathfrak{E}_k)$. We call its transition matrix $P_k^{MC}$. Then,
	it has the form
	\begin{equation*}
		p_{k;\mathfrak{v},\mathfrak{w}}^{MC} = \begin{cases} \dfrac{1}{\mathrm{deg}_k(\mathfrak{v})+k}, & \exists\langle x,y\rangle\in E\text{ s.t. } \mathfrak{v}\triangle\mathfrak{w}=\{x,y\},\vspace{5pt}\\
		\dfrac{k}{\mathrm{deg}_k(\mathfrak{v})+k},&\mathfrak{v}=\mathfrak{w},\\
		0, & \text{otherwise}.\end{cases}
	\end{equation*}
\end{proposition}
The form of the transition matrix follows directly from the construction of $B_t'$ in Figure \ref{fig:bipartite_graph_interpretation_loop} and the uniform distribution over all edges in $B_t'$ as discussed above. The stationary distribution $\pi_k^{MC}$ of $\mathfrak{S}_k^{MC}$ can consequently be derived directly and reversibility follows as well. Based on the transition matrix given by Proposition \ref{prop:exclusion_process_with_loops} and the stationary distribution given by Theorem \ref{thm:stat_dist_S_BE} one can identify $\mathfrak{S}_k^{MC}$ with a random walk on $\mathfrak{L}_k$ where any vertex in $\mathfrak{L}_k$ has additionally to its incident edges $k$ loops. Indeed, for almost all cases of $k$, this is not sufficient to render $\mathfrak{S}_k^{MC}$ a lazy random walk. Indeed, we can quantify this transition.
\begin{lemma}\label{lem:lazy_S_MC_geom_cond}
	The random walk $\mathfrak{S}_k^{MC}$ is lazy if and only if 
	\begin{equation*}
		\min_{\mathfrak{v}\in\mathfrak{V}_k}\mathrm{avg\;deg}_k(L_{\mathfrak{v}}) \geq \bar{d}-1.
	\end{equation*}
\end{lemma}
\begin{proof}
	Recall that the Markov chain is called lazy if and only if
	\begin{equation}
		\min_{\mathfrak{v},\mathfrak{v}}p_{k;\mathfrak{v},\mathfrak{v}}^{MC} \geq \dfrac{1}{2}.
	\end{equation}
	Therefore, using the expression derived in Proposition \ref{prop:exclusion_process_with_loops} we obtain that  $\mathfrak{S}_k^{MC}$ is lazy if and only if for all $\mathfrak{v}\in\mathfrak{V}_k$ we have
	\begin{align*}
		\dfrac{\mathrm{deg}_k(\mathfrak{v})}{k}+1\leq 2 \Leftrightarrow \bar{d} - \mathrm{avg\;deg}_k(L_{\mathfrak{v}}) \leq 1
	\end{align*}
	which proves the claim.
\end{proof}
An obvious property is $\min_{\mathfrak{v}\in\mathfrak{V}_k}\mathrm{avg\;deg}_k(L_{\mathfrak{v}})\leq \bar{d}$ since $L_{\mathfrak{v}}$ is a vertex induced sub-graph and, therefore, the degree of any vertex in $L_{\mathfrak{v}}$ is bounded by $\bar{d}$ which implies the same for the average. Geometrically, the condition given by Lemma \ref{lem:lazy_S_MC_geom_cond} leaves, consequently, only little room for the parameter triple $(\bar{n},\bar{d},k)$.
Due to the intermediate position $\mathfrak{S}_k^{MC}$ takes between the simple random walk on $\mathfrak{L}_k$ and the lazy random walk on this state space, we can in general only use that
\begin{equation*}
	\min_{\mathfrak{v}\in\mathfrak{V}_k}p_{k;\mathfrak{v},\mathfrak{v}}^{MC} \geq \gamma > 0
\end{equation*}
for some $\gamma\in\left(0,\frac{1}{2}\right)$ with a phase transition, if for a parameter triple $(\bar{n},\bar{d},k)$  we have $\min_{\mathfrak{v}\in\mathfrak{V}_k}\mathrm{avg\;deg}_k(L_{\mathfrak{v}}) \geq \bar{d}-1$ which is a geometric condition on $L$. We turn now to the proof of Theorem \ref{thm:convergence_speed_S_MC_evolving_sets}.
\begin{proof}[Proof of Theorem \ref{thm:convergence_speed_S_MC_evolving_sets}]
 	The goal of this proof is to obtain bounds on geometric properties of vertex induced subgraphs of $L$, which allow for a meaningful application of known results from Markov chain theory, as for example presented in \cite{Morris2005}. To begin, let $\gamma:=\left(\bar{d}+1-\min_{\mathfrak{v}\in\mathfrak{V}_k}\mathrm{avg\;deg}_k(L_{\mathfrak{v}})\right)^{-1}$ and consider the following two terms which we call the additive and the multiplicative constant, respectively, given as, firstly,
	\begin{equation*}
		 \dfrac{4(1-\gamma)^2}{\gamma^2}\dfrac{\log\left(\dfrac{2|\mathfrak{E}_k|+k|\mathfrak{V}_k|}{4(\min\{\mathrm{deg}_k(\mathfrak{v}),\mathrm{deg}_k(\mathfrak{w})\}+k)}\right)}{\left(\iota(\mathfrak{L}_k)\frac{2|\mathfrak{E}_k|+k|\mathfrak{V}_k|}{\max_{\mathfrak{v}\in\mathfrak{V}_k}\mathrm{deg}_k(\mathfrak{v})+k}\right)^2}
	\end{equation*}
	and, secondly,
	\begin{equation*}
		 \dfrac{4(1-\gamma)^2}{\gamma^2}\left(\iota(\mathfrak{L}_k)\frac{2|\mathfrak{E}_k|+k|\mathfrak{V}_k|}{\max_{\mathfrak{v}\in\mathfrak{V}_k}\mathrm{deg}_k(\mathfrak{v})+k}\right)^{-2}.
	\end{equation*}
	By \cite{Mohar1989} we can estimate the isoperimetric constant $\iota(\mathfrak{L}_k)$ from below by a constant times the vertex connectivity $\kappa(\mathfrak{L}_k)$ of $\mathfrak{L}_k$ in the form
	\begin{equation*}
		\iota(\mathfrak{L}_k)\geq \dfrac{2}{|\mathfrak{V}_k|}\kappa(\mathfrak{L}_k)
	\end{equation*}
	and we have from Theorem \ref{thm:L_k_vertex_connectivity} the value of $\kappa(\mathfrak{L}_k)$ as $\kappa(\mathfrak{L}_k)=\min_{\mathfrak{v}\in\mathfrak{V}_k}\mathrm{deg}_k(\mathfrak{v})$. This gives us, consequently, in the first case the lower bound for 
	\begin{align*}
		\iota(\mathfrak{L}_k)\frac{2|\mathfrak{E}_k|+k|\mathfrak{V}_k|}{\max_{\mathfrak{v}\in\mathfrak{V}_k}\mathrm{deg}_k(\mathfrak{v})+k} &\geq 2\min_{\mathfrak{v}\in\mathfrak{V}_k}\mathrm{deg}_k(\mathfrak{v})\dfrac{\mathrm{avg\;deg}(\mathfrak{L}_k)+k}{\max_{\mathfrak{v}\in\mathfrak{V}_k}\mathrm{deg}_k(\mathfrak{v})+k}\\
		&\geq  \min_{\mathfrak{v}\in\mathfrak{V}_k}\mathrm{deg}_k(\mathfrak{v})\dfrac{\mathrm{avg\;deg}(\mathfrak{L}_k)}{\max_{\mathfrak{v}\in\mathfrak{V}_k}\mathrm{deg}_k(\mathfrak{v})}
	\end{align*}
	using for the last inequality that $\min_{\mathfrak{v}\in\mathfrak{V}_k}\mathrm{avg\;deg}_k(L_{\mathfrak{v}}) < \bar{d}-1$. This in addition to $\max_{\mathfrak{v}\in\mathfrak{V}_k}\mathrm{deg}_k(\mathfrak{v})\leq k(\bar{n}-k)$ and Proposition \ref{prop:average_degree_quotient_upper_bound} gives the following upper bound for the multiplicative constant by
	\begin{align*}
		\dfrac{4(1-\gamma)^2}{\gamma^2}\left(\frac{\iota(\mathfrak{L}_k)(2|\mathfrak{E}_k|+k|\mathfrak{V}_k|)}{\max_{\mathfrak{v}\in\mathfrak{V}_k}\mathrm{deg}_k(\mathfrak{v})+k}\right)^{-2} &\leq \dfrac{4}{k^2} \left(\dfrac{\max_{\mathfrak{v}\in\mathfrak{V}_k}\mathrm{deg}_k(\mathfrak{v})}{\mathrm{avg\;deg}(\mathfrak{L}_k)}\right)^4\left(\dfrac{\mathrm{avg\;deg}(\mathfrak{L}_k)}{\min_{\mathfrak{v}\in\mathfrak{V}_k}\mathrm{deg}_k(\mathfrak{v})}\right)^2\\
		&\leq  \dfrac{4}{k^2} \left(\dfrac{\bar{n}-1}{\bar{d}}\right)^4\left(\dfrac{{\mathrm{avg\;deg}(\mathfrak{L}_k)}}{\min_{\mathfrak{v}\in\mathfrak{V}_k}\mathrm{deg}_k(\mathfrak{v})}\right)^2.
	\end{align*}
Employing the lower bound $\min_{\mathfrak{v}\in\mathfrak{V}_k}\mathrm{deg}_k(\mathfrak{v})\geq \bar{d}$, which can be derived easily by using $k\leq \bar{n}-1$ and the fact that there is at least one empty vertex in $L$ with degree $\bar{d}$, one can obtain the following upper bound, using Proposition \ref{prop:average_degree_quotient_upper_bound}, independent of the geometry of $L$ given by
	\begin{equation*}
		\dfrac{4}{k^2} \left(\dfrac{\bar{n}-1}{\bar{d}}\right)^4\left(\dfrac{{\mathrm{avg\;deg}(\mathfrak{L}_k)}}{\min_{\mathfrak{v}\in\mathfrak{V}_k}\mathrm{deg}_k(\mathfrak{v})}\right)^2\leq \dfrac{4(\bar{n}-1)^2(\bar{n}-k)^2}{\bar{d}^4}.
	\end{equation*}
	Since the additive constant equals the multiplicative constant times a logarithmic term, we are only going to focus on the logarithmic term and we note
	\begin{align*}
		\log\left(\dfrac{2|\mathfrak{E}_k|+k|\mathfrak{V}_k|}{4(\min\{\mathrm{deg}_k(\mathfrak{v}),\mathrm{deg}_k(\mathfrak{w})\}+k)}\right) &\leq \log\left(\dfrac{2|\mathfrak{E}_k|+k|\mathfrak{V}_k|}{\min_{\mathfrak{v}\in\mathfrak{V}_k}\mathrm{deg}_k(\mathfrak{v})+k}\right)\\
		&=\log\left(\binom{\bar{n}}{k}\right)+\log\left(\dfrac{\mathrm{avg\;deg}(\mathfrak{L}_k)+k}{\min_{\mathfrak{v}\in\mathfrak{V}_k}\mathrm{deg}_k(\mathfrak{v})+k}\right).
	\end{align*}
	For the second summand, we can use the bound which we used already above to obtain
	\begin{equation*}
		\log\left(\dfrac{\mathrm{avg\;deg}(\mathfrak{L}_k)+k}{\min_{\mathfrak{v}\in\mathfrak{V}_k}\mathrm{deg}_k(\mathfrak{v})+k}\right)\leq \log\left(\dfrac{k(\bar{n}-k)}{\bar{n}-1} + \dfrac{k}{\bar{d}}\right).
	\end{equation*}
	Therefore, we arrive at the upper bound on the additive constant
	\begin{equation*}
		\dfrac{4(\bar{n}-1)^2(\bar{n}-k)^2}{\bar{d}^4}\left(\log\left(\binom{\bar{n}}{k}\right)+\log\left(\dfrac{k(\bar{n}-k)}{\bar{n}-1} + \dfrac{k}.{\bar{d}}\right)\right)
	\end{equation*}
	We now employ some results based on evolving sets presented in \cite{Morris2005}. To this end, we will focus on lower bounds for $\Phi_k(u):=\inf\left\{\frac{\partial(S,S^c)}{\pi_k^{MC}(S)}\big|\,\pi_k^{MC}(S)\leq u\right\}$ for $u\in\left[\min_{\mathfrak{v}\in\mathfrak{V}_k}\pi_k^{MC}(\mathfrak{v}),\frac{1}{2}\right]$ as defined in \cite{Morris2005}. In particular, we use that $\Phi(u)\geq \Phi\left(\frac{1}{2}\right)$. The first claim then follows by  Theorem 5 of \cite{Morris2005}, the second one by \cite{LovKan99}. To conclude the proof, note that for $\bar{S}\in \left\{S\subset \mathfrak{V}_k\big|\,|S|\leq \frac{|\mathfrak{V}_k|}{2}\right\}\cap \left\{S\subset \mathfrak{V}_k\big|\, \pi_k^{MC}(S)\leq \frac{1}{2}\right\}$ we have
	\begin{align*}
		\dfrac{\partial(\bar{S},\bar{S}^c)}{\pi_k^{MC}(\bar{S})} \geq \dfrac{\partial(\bar{S},\bar{S}^c)}{|\bar{S}|}\dfrac{2|\mathfrak{E}_k|+k|\mathfrak{V}_k|}{\max_{\mathfrak{v}\in\mathfrak{V}_k}\mathrm{deg}_k(\mathfrak{v})+k}\geq \iota(\mathfrak{L}_k)\dfrac{2|\mathfrak{E}_k|+k|\mathfrak{V}_k|}{\max_{\mathfrak{v}\in\mathfrak{V}_k}\mathrm{deg}_k(\mathfrak{v})+k}
	\end{align*}
	where we used $\bar{S}\in \left\{S\subset \mathfrak{V}_k|\,|S|\leq \frac{|\mathfrak{V}_k|}{2}\right\}$ for the last estimate and the definition of the isoperimetric constant of a graph. Since, additionally, $\bar{S}\in \left\{S\subset \mathfrak{V}_k\big|\, \pi_k^{MC}(S)\leq \frac{1}{2}\right\}$, taking the infimum over all $S\in\left\{S\subset \mathfrak{V}_k\big|\, \pi_k^{MC}(S)\leq \frac{1}{2}\right\}$ we arrive at
	\begin{equation*}
		\Phi_k\left(\frac{1}{2}\right) \geq  \iota(\mathfrak{L}_k)\dfrac{2|\mathfrak{E}_k|+k|\mathfrak{V}_k|}{\max_{\mathfrak{v}\in\mathfrak{V}_k}\mathrm{deg}_k(\mathfrak{v})+k} 
	\end{equation*}
	and, therefore, for all $u\in\left[\min_{\mathfrak{v}\in\mathfrak{V}_k}\pi_k^{MC}(\mathfrak{v}),\frac{1}{2}\right]$ we obtain
	\begin{equation*}
		\Phi_k(u) \geq  \iota(\mathfrak{L}_k)\dfrac{2|\mathfrak{E}_k|+k|\mathfrak{V}_k|}{\max_{\mathfrak{v}\in\mathfrak{V}_k}\mathrm{deg}_k(\mathfrak{v})+k}.
	\end{equation*}
	The second step consists in finding a lower bound for $\min_{\mathfrak{v}\in\mathfrak{V}_k} p_{k;\mathfrak{v},\mathfrak{v}}^{MC}$. We have already found in the proof of Lemma \ref{lem:lazy_S_MC_geom_cond} that
	\begin{equation*}
		\min_{\mathfrak{v}\in\mathfrak{V}_k} p_{k;\mathfrak{v},\mathfrak{v}}^{MC}\geq \dfrac{1}{\bar{d}+1-\min_{\mathfrak{v}\in\mathfrak{V}_k}\mathrm{avg\;deg}_k(L_{\mathfrak{v}})}=\gamma.
	\end{equation*}
	Having found the necessary bounds, we can apply Theorem 5 of \cite{Morris2005} and obtain after integration 
	\begin{equation}\label{eq:upper_bd_integral_cheeger_constant}
		\int_{4(\pi_k^{MC}(\mathfrak{v})\wedge \pi_k^{MC}(\mathfrak{w}))}^{4\varepsilon^{-1}}\dfrac{4\,\mathrm{d}u}{u\Phi_k(u)^2}\leq \dfrac{\left(\log\left(\dfrac{4}{\varepsilon}\right) - \log\left(\frac{4(\min\{\mathrm{deg}_k(\mathfrak{v}),\mathrm{deg}_k(\mathfrak{w})\}+k)}{2|\mathfrak{E}_k|+k|\mathfrak{V}_k|}\right)\right)}{\left(\iota(\mathfrak{L}_k)\frac{2|\mathfrak{E}_k|+k|\mathfrak{V}_k|}{\max_{\mathfrak{v}\in\mathfrak{V}_k}\mathrm{deg}_k(\mathfrak{v})+k}\right)^2}.
	\end{equation}
	We identify the terms in the right hand side of equation \ref{eq:upper_bd_integral_cheeger_constant} with the additive constant and the multiplicative constant for which we have found meaningful upper bounds at the beginning of this proof, such that we obtain the first claim. The second claim follows by the same estimate for $\Phi_k(u)$ and Theorem 2.2 of \cite{LovKan99} as well as equation (5) of \cite{Morris2005} using the bound of the additive term by $\xi(\bar{n},\bar{d},k)$ as defined in the Theorem.
\end{proof}
This concludes our section on the proof of the main Theorem \ref{thm:convergence_speed_S_MC_evolving_sets}.
\section{Outlook}
	The main goal of this document was the presentation of an MCMC algorithm to sample densest $k$ sub-graphs from regular graphs $L$ and the quantification of its convergence speed. We needed the additional condition that $L^c$ is connected, as well. This gives us two natural aspects from which we could start generalizations of the method. \par
Firstly, let us consider the generalization to simple non-regular graphs under the condition that $L^c$ is connected and fix some $k$. We construct the following example $L$ by assuming $k<\lfloor\frac{\bar{n}}{2}\rfloor$ and looking at two complete graphs $K_k$ and $K_{\bar{n}-k}$. We assume that $L$ has $\bar{n}$ vertices and is constructed by gluing together the two complete graphs by adding an edge between one vertex in $K_k$ and one vertex in $K_{\bar{n}-k}$. Therefore, $L^c$ is connected being the complete bipartite graph minus one edge. A densest $k$ sub-graph $L_{\mathfrak{v}}$ is then given by the complete graph $K_k$, which can be realized by placing either $k$ vertices in $K_k$ or $K_{\bar{n}-k}$ calling them $\mathfrak{v}$ and $\mathfrak{v}'$, respectively. Seen as vertices in the Token Graph $\mathfrak{L}_k$ this shows that the identifiability of densest $k$ sub-graphs via their degree in $\mathfrak{L}_k$ is violated if $L$ is not regular because $\mathrm{deg}_k(\mathfrak{v})<\mathrm{deg}_k(\mathfrak{v}')$ even though they have identical density as defined in the introduction. Nonetheless, $\mathrm{deg}_k(\mathfrak{v})=1$ and is the minimal degree. Therefore, the sampling still gives densest subgraphs with high probability but the hit rate will be much lower since not all densest sub-graphs lie in the same level set w.r.t. $\mathrm{deg}_k(\cdot)$. Adjustments of the exclusion process construction can be helpful here, for another example which respects finer geometric features of the sub-graphs see \cite{Fischer2022}.\par
Keeping on the other hand the regularity condition and dropping the connectedness of $L^c$, we encounter another interesting problem. The graph $L^c$ is $\bar{n}-\bar{d}-1$ regular by assumption and assuming that it is disconnected, we find a least dense sub-graph on $k$ vertices by finding a redistribution of particles on the connected components. Note that each connected component is by assumption itself a $\bar{n}-\bar{d}-1$ regular graph. Interesting combinatorial questions arise then from monotony questions of the number of particles in each connected components w.r.t. to its size. Under the condition on $L$, that there is a monotonously increasing function $f_L$ which maps integers $l\in\{1,\hdots,\bar{n}\}$ to the number of particles in $k'\{1,\hdots,k\}$ such that the least dense sub-graph can be obtained by putting $f(l)$ particles in the connected component of $L^c$ of size $l$, our result works as well. But the geometric implications on $L$ as well as interpretations for $f$ need to be worked out before any meaning can be given to a result in this case. 

\newpage
\printbibliography

\newpage
\appendix
\section{Note on the subtleties of the exclusion process construction}
Exclusion processes are not a new subject and have been analyzed, mostly for continuous time instead of discrete time, in statistical mechanics at length. This first subsection is meant to make the point regarding the construction from the first part of this publication. We translate a "classical exclusion process" discussed in \cite{Diaconis1993} into the framework of dynamic random graphs and then to the token graph $\mathfrak{L}_k$. To this end define $\mathfrak{N}_{k;t}=\{v\in V| \eta_{k;t}(v) = 1\}$ and the time dependent graph $B = (B_t)$ with $B_t = ((\mathfrak{N}_{k;t}, V\backslash \mathfrak{N}_{k;t}, \Sigma_{k,t})$.\par
We adjust the step of $\Sigma_{k,t}$ when going from time $t$ to $t+1$ and illustrate it using the same $3$-regular graph on $6$ vertices depicted in Figure \ref{fig:three_reg_graph_underlying}. Now, the discrete time exclusion process discussed in \cite{Diaconis1993} can be understood as follows as a sequence of random graphs. To any $\eta\in\{0,1\}^V$ we associate as before a $\mathfrak{v}_{\eta}\subset V$ with $\mathfrak{v}_{\eta} = \{v\in V|\eta(v) = 1\}$ and in this sense we can write with a little abuse of notation $\eta = \mathfrak{v}_{\eta}$. Recall, that we assume that for some $\mathfrak{v}\subset V$, $|\mathfrak{v}|=k$ at time $t\in\mathbb{N}$ the process satisfies $\eta_{k;t}=\mathfrak{v}$. Set $\mathfrak{v}^c:=V\backslash\mathfrak{v}$. Consider the possibly disconnected bipartite sub-graph $L_{\mathfrak{v},\mathfrak{v}^c}=((\mathfrak{v},\mathfrak{v}^c),E_{\mathfrak{v},\mathfrak{v}^c})$ of $L$. Then, add all edges to $L_{\mathfrak{v},\mathfrak{v}^c}$ which are contained in the vertex induced sub-graph $L_{\mathfrak{v}}=(\mathfrak{v},E_{\mathfrak{v}})$ of $L$. We call the resulting graph $B_t = (V, \Sigma_{k,t})$ with $ \Sigma_{k,t}= E_{\mathfrak{v},\mathfrak{v}^c}\sqcup E_{\mathfrak{v}}$. Indeed, for a $\bar{d}$-regular graph $L$ we obtain $|\Sigma_{k,t}|=\bar{d}k$. Figure \ref{fig:bipartite_graph_interpretation_DSC} illustrates one possible situation based on a $3$-regular graph on six vertices. Remark that all edges incident to vertices in $\mathfrak{v}$ are also present in $B_t$ which preserves their degree and makes it accessible to calculate the number of edges in $B_t$.
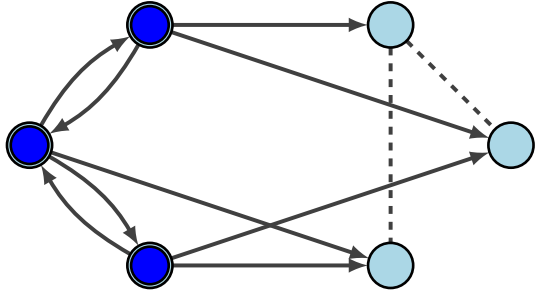
\begin{figure}[!htb]
	\centering
	\begin{tikzpicture}[scale = 0.8]
	    	\Vertex[x=-2,y=0]{a1}
		\Vertex[x=-4,y=2]{a2}
		\Vertex[x=-2,y=4]{a3}
		\Vertex[x=2,y=0]{b1}
		\Vertex[x=4,y=2]{b2}
		\Vertex[x=2,y=4]{b3}
		\Vertex[x=-2,y=0,color = blue,size = 0.5]{z}
		\Vertex[x=-4,y=2,color = blue,size = 0.5]{z}
		\Vertex[x=-2,y=4,color = blue,size = 0.5]{z}
		\Edge[Direct](a1)(b1)
		\Edge[Direct](a1)(b2)
		\Edge[Direct](a2)(b1)
		\Edge[Direct](a3)(b3)
		\Edge[Direct](a3)(b2)
		\Edge[style={-latex}, bend=15](a1)(a2)
		\Edge[style={-latex}, bend=15](a3)(a2)
		\Edge[style={-latex}, bend=15](a2)(a1)
		\Edge[style={-latex}, bend=15](a2)(a3)
		\Edge[style=dashed](b3)(b1)
		\Edge[style=dashed](b3)(b2);
	\end{tikzpicture}
	\caption{For an underlying $3$-regular graph we apply the construction based on the classical exclusion process. The graph $B_t$ constructed from the set of occupied sites $\mathfrak{v}$ containing blue particles and its complement in $V$ colored in pale blue. A particle displacement happens by drawing uniformly one of the directed edges and exchanging the state of the vertices connected by said edge.\label{fig:bipartite_graph_interpretation_DSC}}
\end{figure}
The exclusion process now consists of drawing one edge uniformly from $\Sigma_{k,t}$ and exchanging the state of the endpoints. Note that this might lead to exchanging two occupied sites which renders $P^{\eta}$ independent of $\eta$. As an interpretation of the exclusion process in this case one can see the exchange of states of two particles as their collision, both jumping back to the state they came from. Note that, as discussed in \cite{Diaconis1993}, results in a process whose associated Markov chain on $\mathfrak{L}_k$ is ergodic with a uniform stationary distribution over all states in $\mathfrak{L}_k$. It, therefore, does not "see" different sub-structures of the graph $L$ and behaves identically on all regular graphs with fixed degree $\bar{d}$. It is, consequently, not adapted to the problem we want to approach. Further structures of the graphs $B_t$ and $B_t'$, presented in hereinabove, can be imagined, accentuating different graphs structures in need of being analyzed via MCMC approaches.  
\section{Necessary results on Token Graphs}
\begin{theorem}[\cite{FischerToken2022}]\label{thm:L_k_vertex_connectivity}
	Let $L$ be a connected simple graph on $\bar{n}$ vertices and let $k\in\{1,\hdots,\bar{n}-1\}$. Then, $\kappa(\mathfrak{L}_k) = \min_{\mathfrak{v}\in\mathfrak{V}_k}\mathrm{deg}_k(\mathfrak{v})$ and a corresponding vertex cut is the neighborhood of $\mathfrak{v}_{\ast}\in\mathfrak{V}_k$ with $\mathrm{deg}_k(\mathfrak{v}_{\ast}) = \min_{\mathfrak{v}\in\mathfrak{V}_k}\mathrm{deg}_k(\mathfrak{v})$. 
\end{theorem}
\begin{proposition}[\cite{FischerToken2022}]\label{prop:size_edge_set_E_k}
	Consider the graph $\mathfrak{L}_k=(\mathfrak{V}_k, \mathfrak{E}_k)$ for $k\in\{1,\hdots,\bar{n}-1\}$. Then $|\mathfrak{V}_k|=\binom{\bar{n}}{k}$ and
	\begin{equation}\label{eq:size_of_edge_set_L_k_short}
		|\mathfrak{E}_k|=\dfrac{1}{2}\left(\bar{d}k\binom{\bar{n}}{k}-\bar{n}\bar{d}\binom{\bar{n}-2}{k-2}\right)=k(\bar{n}-k)\binom{\bar{n}}{k}\dfrac{\bar{d}}{2(\bar{n}-1)}.
	\end{equation}
\end{proposition}
\begin{proposition}[\cite{FischerToken2022}]\label{prop:average_degree_quotient_upper_bound}
	Let $L$ be a simple connected graph on $\bar{n}$ vertices and $k\in\{1,\hdots,\bar{n}-1\}$. Denote by $\mathrm{avg\;deg}(\mathfrak{L}_k)$ the average degree in $\mathfrak{L}_k$, i.e., 
	\begin{equation*}
		\mathrm{avg\;deg}(\mathfrak{L}_k):=\dfrac{1}{|\mathfrak{V}_k|}\sum_{\mathfrak{v}\in\mathfrak{V}_k}\mathrm{deg}_k(\mathfrak{v}).
	\end{equation*}
	and by $\bar{d}$ the average degree in $L$.	Then average degree in $\mathfrak{L}_k$ satisfies
	\begin{equation}\label{eq:expression_quotient_average_degree}
		\dfrac{\mathrm{avg\;deg}(\mathfrak{L}_k)}{k(\bar{n}-k)}=\dfrac{\bar{d}}{\bar{n}-1}\leq 1.
	\end{equation}
\end{proposition} 

\end{document}